\documentclass[12pt]{amsart}
\usepackage[dvips]{graphicx}
\usepackage[all]{xypic}
\usepackage{amsmath,graphics}
\usepackage{amsfonts,amssymb}

\theoremstyle{plain}
\newtheorem*{theorem*}{Theorem}
\newtheorem*{lemma*} {Lemma}
\newtheorem*{corollary*} {Corollary}
\newtheorem*{proposition*} {Proposition}
\newtheorem{theorem}{Theorem}[section]
\newtheorem{lemma}[theorem]{Lemma}

\newtheorem{proposition}[theorem]{Proposition}

\theoremstyle{remark}
\newtheorem*{remark}{Remark}
\newtheorem*{definition}{Definition}

\theoremstyle{definition}

\def\ug{{\mathcal{U}(\G)}}
\def\kg{\K(\G)}
\def\zg{\Z[\G]}
\def\ng{{\mathcal{N}(\G)}}

\def\eps{\epsilon}

\def\R{\Bbb{R}}

\def\K{\Bbb{K}}

\def\id{\mbox{id}}

\def\Z{\Bbb{Z}}
\def\C{\Bbb{C}}
\def\xc{X(\scrc)}
\def\tixc{\widetilde{X(\scrc)}}
\def\tibxc{\widetilde{\yc}}

 \def\yc{Y(\scrc)}
\def\N{\Bbb{N}}

\def\part{\partial}

\def\B{\Bbb{B}}
\def\BB{\mathcal{B}}

\def\bp{\begin{pmatrix}}

\def\sm{\setminus}
\def\ep{\end{pmatrix}}
\def\bn{\begin{enumerate}}

\def\en{\end{enumerate}}
\def\ba{\begin{array}}
\def\ea{\end{array}}

\def\S{\Sigma}

\def\a{\alpha}
\def\b{\beta}
\def\ti{\tilde}

\def\fr12{\frac{1}{2}}

\def\im{\mbox{Im}}

\def\be{\begin{equation}}
\def\ee{\end{equation}}

\def\inter{\mbox{int}}

\def\K{\Bbb{K}}

\def\G{\Gamma}

\def\K{\Bbb{K}}

\def\v{\varphi}

\def\scrc{\mathcal{C}}
\def\scrd{\mathcal{D}}
\def\scra{\mathcal{A}}

\def\sing{\mbox{Sing}}

\def\cmtbf#1{} \def\cmt#1{}

\begin{document}

\title{$L^2$--Betti numbers of plane algebraic curves}
\author{Stefan Friedl, Constance Leidy and Laurentiu Maxim}
\address{\begin{tabular}{l} Universit\'e du Qu\'ebec \`a Montr\'eal, Montr\'eal, Qu\'ebec, Canada and\\ University of Warwick, Coventry, UK\end{tabular}}
\email{sfriedl@gmail.com}
\address{        Department of Mathematics and Computer Science, 655 Science Tower, Wesleyan Univ., 265 Church St., Middletown, CT 06459, USA. }
\email{cleidy@wesleyan.edu}
\address{\begin{tabular}{l}
Institute of Mathematics of the Romanian Academy, P.O. Box 1-764, \\
\quad 70700 Bucharest, Romania.\\
Department of Mathematics \& Computer Science, CUNY-Lehman College,\\
\quad 250 Bedford Park Blvd West, Bronx, NY 10468, USA
\end{tabular}}
\email{laurentiu.maxim@lehman.cuny.edu}

\def\subjclassname{\textup{2000} Mathematics Subject Classification}
\expandafter\let\csname subjclassname@1991\endcsname=\subjclassname \expandafter\let\csname
subjclassname@2000\endcsname=\subjclassname \subjclass{Primary  32S20; 32S55; Secondary 57M25; 14F17} \keywords{$L^2$--Betti numbers,
curve complements, singularities }
\date{\today}

\begin{abstract}
In \cite{DJL07} it was shown that if $\scra$ is an affine hyperplane arrangement in $\C^n$, then at most one of the $L^2$--Betti numbers $b_p^{(2)}(\C^n\sm  \scra,\id)$  is non--zero. We will prove an analogous statement
for complements of any algebraic curve in $\C^2$. Furthermore we also recast and extend results of \cite{LM06} in terms of $L^2$--Betti numbers.
\end{abstract}

\maketitle

\section{Introduction}
Let $X$ be any  topological space and $\varphi:\pi_1(X)\to \G$ a homomorphism to a  group (all groups are assumed countable).
Then for $p\in \N\cup \{0\}$ we can consider the $L^2$--Betti number $b_p^{(2)}(X,\varphi)\in [0,\infty]$.
We recall the definition and some of the most important properties of $L^2$--Betti numbers in Section \ref{section:l2betti}.


Let $\scrc\subset \C^2$ be a reduced plane algebraic curve with
irreducible components $\scrc_1,\dots,\scrc_r$.
We write $X(\scrc):=\C^2\sm \nu \scrc$, for $\nu \scrc$ a regular
neighborhood of $\scrc$ inside $\C^2$. We denote the meridians about
the nonsingular parts of $\scrc_1,\dots,\scrc_r$ by
$\mu_1,\dots,\mu_r$. Note that these meridians come with a preferred
orientation since the non-singular parts of the irreducible
components $\scrc_i$ are complex submanifolds of $\C^2$.

It is well--known (cf. Theorem \ref{theorem:top}) that $H_1(X(\scrc);\Z)$ is the free abelian group generated by the meridians $\mu_1,\dots,\mu_r$. Throughout the paper we denote by $\phi$ the map $\pi_1(X(\scrc);\Z)\to \Z$ given by sending each meridian $\mu_i$ to $1$. We also refer to $\phi$ as the total linking homomorphism.
We henceforth call a homomorphism $\a:\pi_1(\xc)\to \G$ to a group \emph{admissible} if the total linking homomorphism $\phi$ factors through $\a$.

Our first result is the following.
\begin{theorem} \label{mainthm}
Let $\scrc\subset \C^2$ be a reduced algebraic curve $\scrc$ whose
projective completion intersects the line at infinity transversely.
Let $\a:\pi_1(X(\scrc))\to \G$ be an admissible homomorphism, then
\[ b_p^{(2)}(X(\scrc),\a)= \left\{ \ba{rl} 0, & \mbox{ for }p\ne 2, \\ \chi(X(\scrc)),
&\mbox{ for }p=2. \ea \right.\]
\end{theorem}

In \cite{DJL07} it was shown that if $\scra$ is an affine hyperplane
arrangement in $\C^n$, then at most one of the $L^2$--Betti numbers
$b_p^{(2)}(\C^n\sm  \scra,\id)$  is non--zero. Theorem \ref{mainthm} can be
seen as  an analogous statement for the complement of an algebraic
curve in $\C^2$ which is in general position at infinity. Note that
in the case that $\G$ is a polytorsion--free--abelian (PTFA) group,
then this theorem, together with Proposition \ref{prop:b2kg},
recovers \cite[Corollary~4.2]{LM06}.

Given an algebraic curve $\scrc$ we denote by $\tixc$ the infinite
cyclic cover of $\xc$ corresponding to $\phi$. Given an admissible
homomorphism $\a:\pi_1(\xc)\to \G$ we let
$\ti{\G}:=\im\{\pi_1(\tixc)\to \pi_1(\xc)\xrightarrow{\a}\G\}$ and
we denote the induced map $\pi_1(\tixc)\to \ti{\G}$ by $\ti{\a}$.
We will now study the invariant
\[ b_1^{(2)}(\tixc,\ti{\a}:\pi_1(\tixc)\to \ti{\G}).\]
The idea of looking for invariants of the fundamental group of the
complement that capture information about the topology of the curve
goes back to the early work of Zariski, and was further developed by
A. Libgober by analogy with the classical knot theory (cf.
\cite{Lib82,Lib83,Lib92,Lib01}). 
In particular Libgober studied the ordinary one--variable Alexander polynomial corresponding to $\xc$,
its degree is given by the the ordinary Betti number of
$\tixc$ (cf. e.g. \cite[p.~368]{Co04}).  In that sense the study of the $L^2$--Betti numbers of $\tixc$
can be seen as a non--commutative generalization of the approach of Libgober. 

Following work of Cochran and Harvey the second and third author consider in
 \cite{LM06} the  the following homomorphism
\[ \pi_n: \pi_1(\xc)\to \pi_1(\xc)/\pi_1(\xc)_r^{(n+1)}=:\G_n,\]
where given a group $G$ we denote by $G_r^{(n)}$ the $n$--th term in the rational derived series
(cf. \cite{Ha05}). The group $\G_n$ is a PTFA group and the authors define an invariant
$\delta_n(\scrc)$ as the dimension of the first homology of $\tixc$ with coefficients in the skew field
associated to $\ti{\G}_n$. Some of these invariants are  computed in \cite{LM06} and \cite{LM07}. The main result of \cite{LM06} gives upper bounds on $\delta_n(\scrc)$ in terms of information coming from the singularities of $\scrc$.

We will see in Theorem \ref{thm:dimb2} that
\[ \delta_n(\scrc)=b_1^{(2)}(\tixc,\ti{\pi}_n:\pi_1(\tixc)\to \ti{\G}_n).\]
The following theorem can therefore be viewed as a generalization of \cite[Theorem~4.1]{LM06}.
Note that for the invariants $\delta_n(\scrc)$ it gives a slightly better bound than \cite[Theorem~4.1]{LM06}.

\begin{theorem} \label{mainthm2}
Let $\scrc\subset \C^2$ be a reduced plane algebraic curve of degree
$d$ whose projective completion intersects the line at infinity
transversely. Denote the set of singular points by $P_1,\dots,P_s$,
and for a singular point $P_i$ denote by $\mu(\scrc,P_i)$ the
associated Milnor number of the singularity germ at $P_i$.  Let
$\a:\pi_1(X(\scrc))\to \G$ be an admissible homomorphism, then
\[ b_1^{(2)}(\tixc,\ti{\a}:\pi_1(\tixc)\to \ti{\G})\leq \sum_{i=1}^s (\mu(\scrc,P_i)+n_i-1)+2g+d.\]
Here $n_i$ denotes the number of branches through $P_i$ and $g$ is the genus of the normalization of the projective completion of $\scrc$.
\end{theorem}

This theorem shows that the topology of the singularities imposes
restrictions on the $L^2$--Betti numbers of the curve complement. In
this sense this result is in the same vein as the results of
Libgober \cite{Lib82} and Cogolludo--Florens \cite{CF07}, but see
also \cite{Lib94,DM07,Ma06} for similar results in the
higher-dimensional case.


\section{$L^2$--Betti numbers} \label{section:l2betti}

\subsection{The von Neumann algebra and its localizations}

Let $\G$ be a countable group. Define $l^2(\G):=\{ f:\G\to \C \, | \, \sum_{g\in \G}  |f(g)|^2<\infty \}$,
this is a Hilbert space. Then $\G$ acts on $l^2(\G)$ by right multiplication, i.e. $(g\cdot f)(h)=f(hg)$.
This defines an injective map $\C[\G]\to \BB(l^2(\G))$, where $\BB(l^2(\G))$ is the set of bounded operators on $l^2(\G)$.
We henceforth view $\C[\G]$ as a subset of $\BB(l^2(\G))$.

Now define the \emph{von Neumann algebra} $\ng$ to be the closure of $\C[\G]\subset \BB(l^2(\G))$ with respect to pointwise convergence in
$\BB(l^2(\G))$. Note that any $\ng$--module $M$ has a dimension $\dim_{\ng}(M)\in \R_{\geq 0}\cup \{\infty\}$. We refer to \cite[Definition~6.20]{Lu02} for details.

\subsection{The definition of $L^2$--Betti numbers}

Let $X$ be a topological space (not necessarily compact) and let
$\v:\pi_1(X)\to \G$ be a homomorphism to a group. Denote the
covering of $X$ corresponding to $\v$ by $\ti{X}$. Then we can study
the $\ng$--chain complex
\[ C_*^{sing}(\ti{X})\otimes_{\Z[\G]}\ng,\]
where $C_*^{sing}(\ti{X})$ is the singular chain complex of $\ti{X}$
with right $\G$--action given by covering
translation.  Furthermore $\G$
acts canonically on $\ng$ on the left. The $p$--th $L^2$--Betti
number is now defined as
\[ b_p^{(2)}(X,\v):=\dim_{\ng}(H_p(C_*^{sing}(\ti{X})\otimes_{\Z[\G]}\ng))\in [0,\infty].\]
We refer to \cite[Definition~6.50]{Lu02} for more details.

In the following lemma we summarize some of the properties of $L^2$--Betti numbers.
We refer to  \cite[Theorem~6.54,~Lemma~6.53~and~Theorem~1.35]{Lu02}  for the proofs.

\begin{lemma} \label{lem:propb2}
Let $X$ be a topological space  and let $\v:\pi_1(X)\to \G$ be a homomorphism to a group.
\bn
\item $b_p^{(2)}(X,\varphi)$ is a homotopy invariant of the pair $(X,\varphi)$.
\item $b_0^{(2)}(X,\varphi)=0$ if $\im(\v)$ is infinite and $b_0^{(2)}(X,\varphi)=\frac{1}{|Im(\v)|}$ if $\im(\v)$ is finite.
\item If $X$ is a finite CW--complex, then
\[ \sum_p \,(-1)^p \,b_p^{(2)}(X,\v)=\chi(X),\] where $\chi(X)$ denotes the Euler characteristic of $X$.
\item If $\im(\varphi)\subset \ti{\G}\subset \G$, then $b_p^{(2)}(X,\varphi:\pi_1(X)\to \ti{\G})=b_p^{(2)}(X,\varphi:\pi_1(X)\to \G)$.
\en
\end{lemma}

We will also make use of the following lemma.

\begin{lemma}\label{lem:b2surj}
Let $f:Y\to Z$ be a map of topological spaces such that $\pi_1(Y)\to \pi_1(Z)$ is surjective.
Assume that we are given a homomorphism $\b:\pi_1(Z)\to \G$. Then
\[ b_1^{(2)}(Y,\pi_1(Y)\xrightarrow{f_*}\pi_1(Z)\xrightarrow{\b}\G) \geq b_1^{(2)}(Z,\b).\]
\end{lemma}

\begin{proof}
We denote the homomorphism
$\pi_1(Y)\xrightarrow{f_*}\pi_1(Z)\xrightarrow{\b}\G$ by $\b$ as
well. Note that an Eilenberg--Maclane space $K$ for $\pi_1(Z)$ is
given by adding handles of degree greater than 2 to $Z$. In
particular $b_1^{(2)}(Z,\b)=b_1^{(2)}(K,\b)$. By the homotopy
invariance of the $L^2$--Betti numbers we know that for any other
Eilenberg--Maclane space for $\pi_1(Z)$ we get the same invariant.

Since $f_*:\pi_1(Y)\to \pi_1(Z)$ is surjective we can also build an
Eilenberg--Maclane space $K'$ for $\pi_1(Z)$ by adding handles of
degree greater or equal than 2 to $Y$. By the above discussion we
therefore get
\[ b_1^{(2)}(Z,\b)=b_1^{(2)}(K,\b)=b_1^{(2)}(K',\b).\]

It now remains to show that $b_1^{(2)}(Y,\b)\geq  b_1^{(2)}(K',\b)$. Since $K'$ is given by adding handles of degree greater or equal than 2 to $Y$ we get the following commutative diagram
\[ \ba{cccccccccccccccccccc}
C_2(Y;\ng)&\to &C_1(Y;\ng)&\to&C_0(Y;\ng)&\to& 0\\[0.1cm]
\downarrow &&\downarrow =&&\downarrow = \\[0.1cm]
P\oplus C_2(Y;\ng)&\to &C_1(Y;\ng)&\to&C_0(Y;\ng)&\to& 0\\[0.1cm]
\parallel&&\parallel&&\parallel \\[0.1cm]
C_2(K';\ng)&\to &C_1(K';\ng)&\to&C_0(K';\ng)&\to& 0.\ea \]
where $P$ is the free $\ng$--module generated by the extra 2--handles of $K'$.
This shows that the map $H_1(Y;\ng)\to H_1(K';\ng)$ is surjective. But then the claim on $L^2$--Betti numbers follows immediately from \cite[Theorem~6.7]{Lu02}.
\end{proof}

\subsection{The $L^2$--Betti numbers and the Cochran--Harvey invariants}

Recall that a group $\G$ is called \emph{locally indicable} if for
every finitely generated non--trivial subgroup $H\subset \G$ there
exists an epimorphism $H\to \Z$. We will also need the notion of an
\emph{amenable} group. We refer to  \cite[p.~256]{Lu02} for the
definition of an amenable group, but note that any solvable group is
amenable and that  groups containing the free group on two generators are not amenable. In the following we refer to a locally indicable
torsion--free amenable group as a LITFA group.

Denote by $S$ the set of non--zero divisors
of the ring $\ng$. By \cite[Proposition~2.8]{Re98}
(see also \cite[Theorem~8.22]{Lu02}) the pair $(\ng,S)$
satisfies the right Ore condition.
We now let $\ug:=\ng S^{-1}$, this ring is called the
\emph{algebra of operators affiliated to $\ng$}. For any
$\ug$--module $M$ we also have a dimension $\dim_{\ug}(M)$. By
\cite[Theorem~8.31]{Lu02} we have
\[ b_p^{(2)}(X,\v)=\dim_{\ug}(H_p(C_*^{sing}(\ti{X})\otimes_{\Z[\G]}\ug)).\]
We collect some properties of LITFA groups in the following well--known theorem.

\begin{theorem}\label{thm:tfs}
Let $\G$ be a LITFA group.
\bn
\item All non--zero elements in $\Z[\G]$ are non--zero divisors in $\ng$.
\item  $\Z[\G]$  is an Ore domain and embeds in its classical right ring of
quotients $\K(\G)$.
\item $\kg$ is flat over $\Z[\G]$.
\item There exists a monomorphism $\kg\to \ug$ which makes the following diagram commute
\[ \xymatrix{\Z[\G] \ar[r] \ar[dr] & \kg\ar[d] \\ &\ug.}\]
\en
\end{theorem}

\begin{proof}
The first claim follows from results of Linnell \cite{Lin92} and
Burns and Hale \cite{BH72}. Note that it implies in particular that
all non--zero elements in $\zg$ are non--zero divisors in $\zg$. The second
part now follows from  \cite[Corollary~6.3]{DLMSY03}. The third part
is a well--known property of Ore localizations (cf. e.g.
\cite[p.~99]{Re98}). Finally the last statement follows from the
definitions of $\kg$ and $\ug$ as Ore localizations and the fact
that $\Z[\G]\sm \{0\}\subset S$.
\end{proof}

We recall that
a group $\G$ is called poly--torsion--free--abelian (PTFA) if there exists a normal series
\[ 1=\G_0 \subset \G_1\subset \dots \subset \G_{n-1}\subset \G_n=\G \]
such that $\G_{i}/\G_{i-1}$ is torsion free abelian.
 PTFA groups played an important role in several recent
papers like \cite{COT03}, \cite{Co04}, \cite{Ha05} and \cite{LM06}.

It is easy to see that  PTFA groups are LITFA. Note that the quotients $\pi/\pi_r^{(n)}$ of a group by terms in the rational derived series are PTFA (cf. \cite{Ha05}).
The following proposition relates $L^2$--Betti numbers to ranks of modules over skew fields.
It seems to be well--known (cf. for example
\cite[p.~8]{Ha08}), but for the sake of completeness we quickly outline the
proof.

\begin{proposition} \label{prop:b2kg}
Let $\v:\pi_1(X)\to \G$ be a homomorphism to a
LITFA group $\G$. Then we have
\[  b_p^{(2)}(X,\varphi)=\dim_{\K(\G)}(H_p(X;\K(\G)). \]
\end{proposition}

\begin{proof}
By Theorem \ref{thm:tfs} we have an inclusion $\kg\to \ug$.
Since $\kg$ is a skew--field any $\kg$--module is free.
We deduce that $\ug$ is flat as a $\kg$--module.
 In particular if
$d=\dim_{\K(\G)}(H_p(X;\K(\G))<\infty$, then we see that
\[ \ba{rcl} \dim_{\ug}(H_p(X;\ug))&=&\dim_{\ug}(H_p(X;\kg)\otimes_{\kg}\ug)\\
&=&\dim_{\ug}(\kg^d\otimes_{\kg}\ug)\\
&=&
\dim_{\ug}(\ug^d)=d.\ea \]
The case that $d=\dim_{\K(\G)}(H_p(X;\K(\G))=\infty$ follows similarly.
\end{proof}

We now recall the definition of the Cochran--Harvey invariants (which in this context were first studied in \cite{LM06}).  Let $\scrc$ be an algebraic curve in $\C^2$.
Furthermore let
$\a:\pi_1(X(\scrc))\to \G$ be an admissible homomorphism to a LITFA group. Recall that admissible means that
 there exists a map $\phi:\G\to \Z$ such that the following diagram commutes
\[ \xymatrix{ \pi_1(\C^2\sm \scrc)\ar[rr]^{\a}\ar[dr]^\phi && \G\ar[dl]_{\phi}\\
 &\Z.&}\]
Also recall that we denote by $\ti{\G}$ the kernel of $\phi:\G\to \Z$ and that we denote the induced homomorphism $\pi_1(\tixc)\to \ti{\G}$ by $\ti{\a}$.

Now consider the homomorphism $\pi_1(\xc)\to
\pi_1(\xc)/\pi_1(\xc)_r^{(n+1)}=\G_n$. It is easy to see that this
homomorphism is admissible. As in \cite{LM06} we now define
\[ \delta_n(\scrc)=\dim_{\K(\ti{\G}_n)}(H_1(\tixc;\K(\ti{\G}_n)).\]
The following theorem, which is an immediate corollary to Proposition \ref{prop:b2kg},  now shows that the $L^2$--Betti numbers considered in this paper can be viewed as a generalization of the Cochran--Harvey invariants of plane algebraic curves.

\begin{theorem} \label{thm:dimb2}
Let $\scrc\subset \C^2$ be an algebraic curve $\scrc$ and let $\a:\pi_1(X(\scrc))\to \G$ be an admissible homomorphism to a LITFA group.
Then
\[ \dim_{\K(\ti{\G})}(H_1(\tixc;\K(\ti{\G}))=b_1^{(2)}(\tixc,\ti{\a}:\pi_1(\tixc)\to \ti{\G}).\]
\end{theorem}

\section{Proof of Theorem \ref{mainthm} and Theorem \ref{mainthm2}}

\subsection{Plane algebraic curves and their topology} \label{section:defs}
From now on let $\scrc\subset \C^2$ be an algebraic curve with
irreducible components $\scrc_1,\dots,\scrc_r$.  Recall that we
write $\xc=\C^2\sm \nu \scrc$. We now also write
$\yc=\partial(\overline{\nu
\scrc})=\partial(\overline{X(\scrc)})$. Note that $\yc\subset
\xc$.
The following summarizes some well--known results on the topology of $\xc$.

\begin{theorem} \label{theorem:top}
\bn
\item $\pi_1(X(\scrc))$ is normally generated by the meridians about the non--singular parts of the irreducible components and $H_1(X(\scrc);\Z)$
is a free abelian group of rank $r$ with basis given by these
meridians.
\item $X(\scrc)$ is homotopy equivalent to a 2--complex.
\item If $\scrc$ intersects the line at infinity transversely, then $\pi_1(\yc)\to \pi_1(X(\scrc))$ is surjective.
\en
\end{theorem}

\begin{proof}
The first statement follows from the fact that by gluing in  disks
at the meridians we kill the fundamental group. The statement about
the first homology group follows from Lefschetz duality (cf.
\cite[p.~835]{Lib82} or \cite[p.~103]{Di92}). The second statement follows since $X(\scrc)$
has the homotopy type of a $2$-dimensional complex affine variety
(cf. also \cite[Theorem~1.6.8]{Di92} or \cite[Theorem~7.2]{Mi63}).
The last statement follows from applying the Lefschetz hyperplane
theorem (cf. e.g. \cite[p.~25]{Di92}), and by an argument similar to
that of \cite[Proof~of~Theorem~4.1]{LM06}.
\end{proof}

\subsection{Proof of Theorem \ref{mainthm}}
From now on assume that the algebraic curve $\scrc$  intersects the line at infinity transversely. Let $\a:\pi_1(X(\scrc))\to \G$ be an admissible homomorphism.

Since $\G$ is infinite, and since $X(\scrc)$ is homotopy equivalent to a 2--complex we get from Lemma \ref{lem:propb2} that
$b_p^{(2)}(X(\scrc),\a)=0$ for $p=0$ and $p>2$. It therefore remains to show that
$b_1^{(2)}(X(\scrc),\a)=0$. The statement on $b_2^{(2)}(X(\scrc),\a)$ then follows immediately from
Lemma \ref{lem:propb2} (3).
We denote the homomorphism $\pi_1( \yc)\to \pi_1(\xc)\xrightarrow{\a}\G$ by $\a$ as well.
By Theorem \ref{theorem:top} (3) and Lemma \ref{lem:b2surj}  it is enough to prove that
$b_1^{(2)}(\yc,\a)=0$.

Let $\B^4\subset \C^2$ be a sufficiently large closed ball, in the
sense that $\inter(\B^4)\sm (\scrc\cap \inter(\B^4))$ is
diffeomorphic to $\C^2\sm \scrc$. Such a ball exists by
\cite[Theorem 1.6.9]{Di92}. Note that in particular all
singularities of $\scrc$ lie in the interior of $\B^4$.
By the homotopy invariance of the $L^2$--Betti numbers we can abuse the notion and we therefore denote $\B^4\cap \xc$ and $\B^4\cap \yc$ by $\xc$ and $\yc$ again.

 Given a point $P=(x_P,y_P)\in \C^2$ and $\eps>0$ we
write $\B^4(P,\eps)=\{ (x,y)\in \C^2\, |\, |x-x_P|^2+|y-y_P|^2 \leq
\eps^2\}$ and $S^3(P,\eps)=\partial \B^4(P,\eps)$. Now let
$\sing(\scrc):=\{P_1,\dots,P_s\}$ denote the set of singularities of
$\scrc$. Then there exist $\eps_1,\dots,\eps_s>0$ such that \bn
\item $\B^4(P_i,\eps_i)$ are pairwise disjoint,
\item $\B^4(P_i,\eps_i)\subset \mbox{int}(\B^4)$,
\item $\B^4(P_i,\eps_i)\sm \big( \scrc \cap \B^4(P_i,\eps_i)\big)$ is the cone on $S^3(P_i,\eps_i)\sm \big(\scrc\cap S^3(P_i,\eps_i)\big)$.
\en
Such $\eps_i$ exist by Thom's first isotopy lemma (cf. \cite[Section~5]{Di92} for details).
For $i=1,\dots,s$ we write $S_i^3=\partial(\B^4(P_i,\eps_i))$, $L_i:=S_i^3\cap \scrc$ and
 $X(L_i):=S^3_i\sm \nu L_i$.

Let $T_i, i=1,\dots,s$ be the boundaries of $S^3_i\sm \nu L_i$.
These are unions of tori and we denote the connected components of
$T_i$ by $T_i^1,\dots,T_i^{n_i}$.
Let $F_j:=\scrc_j\sm \big(\cup \inter(\B_i^4) \cap \scrc_j\big)$ for $j=1,\dots,r$. Then
$F_1,\dots,F_r$ are the connected components of $F:=\scrc \cap (
\C^2\sm \cup_{i=1}^s \inter(\B_i^4))$. We write $Y(F)=\yc \cap (\C^2\sm
\cup_{i=1}^s \inter(\B_i^4))$ and we denote the connected components of $Y(F)$ by $Y(F_1),\dots,Y(F_r)$. We can therefore
decompose
\[ \yc=\bigcup_{i=1,\dots,r} Y(F_i)\cup_{T_1\cup \dots \cup T_s} \bigcup_{i=1,\dots,s}
X(L_i) .\]

We need the following definition.

\begin{definition}
Let $M$ be a 3--manifold and $\psi\in H^1(M;\Z)$. We say that $(M,\psi)$ fibers
over $S^{1}$ if the homotopy class of maps $M \to S^1$ determined by $\psi \in H^{1}(M;\Z) =
[M,S^{1}]$ contains a representative that is a fiber bundle over $S^{1}$.
\end{definition}

Milnor \cite[Theorem~4.8]{Mi68} showed that for $i=1,\dots,s$ the
pair $(X(L_i),\phi_i)$ fibers over $S^1$, where
$\phi_i:H_1(X(L_i);\Z)\to \Z$ is induced by the (local) total
linking number homomorphism, i.e., by sending all meridians (with
the induced orientation) about the components of $L_i$ to $1$ (e.g.,
see \cite{Di92}, p.~76--77).
Note that $\phi_i$ is precisely the homomorphism given by
homomorphism
\[ \pi_1(X(L_i))\to \pi_1(\yc)\to \pi_1(\xc)\xrightarrow{\phi} \Z.\]

For $i=1,\dots,r$ we now consider $Y(F_i)$. Picking a trivialization of the normal bundle of $F_i$ we can identify $Y(F_i)$ with $F_i\times S^1$. Consider
the homomorphism
\[ \psi_i:\pi_1(F_i\times S^1)\to \pi_1(\yc)\to \pi_1(\xc)\xrightarrow{\phi} \Z.\]
Since the homomorphism $\pi_1(S^1)\to \pi_1(F_i\times
S^1)\xrightarrow{\psi_i}\Z$ is surjective
it is well--known that $(F_i\times
S^1,\psi_i)$ fibers over $S^1$ and the fiber is diffeomorphic to
$F_i$.
It follows from the above discussion that the fibrations $F_i\times
S^1\to S^1$ and $X(L_i)\to S^1$ when restricted to the tori $T_i^j$
correspond to the same classes in $H^1(T_i^j;\Z)$. Since fibrations of
a torus which lie in the same cohomology class are isotopic it
follows that we can glue the fibrations $F_i\times S^1\to S^1$ and
$X(L_i)\to S^1$ to get a fibration $\pi:\yc\to S^1$ such that $\pi_*:\pi_1(\yc)\to \pi_1(S^1)=\Z$ equals $\pi_1(\yc)\to \pi_1(\xc)\xrightarrow{\phi}\Z$.

We now recall the following theorem (\cite[Theorem~1.39]{Lu02}).

\begin{theorem} Let $M$ be a compact 3--manifold and $\psi\in H^1(M;\Z)$ such that $(M,\psi)$ fibers over $S^1$.
If $\b:\pi_1(M)\to G$ is a homomorphism to a group $G$ such that $\psi$ factors through $\b$, then
$ b_p^{(2)}(M,\b)=0$
for all $p$.
\end{theorem}

Since $\a$ is admissible it follows now that $b_p^{(2)}(\yc,\a)=0$.
This concludes the proof of Theorem \ref{mainthm}.

\subsection{Proof of Theorem \ref{mainthm2}}
Let $\scrc\subset \C^2$ be a reduced algebraic curve in general
position at infinity. We pick $\B^4$ as in the previous section, again we abuse the notation and we denote $\B^4\cap \xc$ and $\B^4\cap \yc$ by $\xc$ and $\yc$.

Let $\a:\pi_1(X(\scrc))\to \G$ be an
admissible homomorphism. Denote the induced map $\pi_1(\yc)\to
\pi_1(\xc)\xrightarrow{\phi} \Z$ by $\phi'$. Note that by Theorem \ref{theorem:top} (3)  the map $\phi'$ is
surjective as well. Now denote by $\tibxc$ and $\tixc$ the  infinite
cyclic covers corresponding to $\phi'$ and $\phi$. It follows easily
that the induced map
\[ \pi_1(\tibxc)\to \pi_1(\tixc) \]
is still surjective.
But by   Lemma \ref{lem:b2surj}  we then also have
\[ b_1^{(2)}(\tibxc,\ti{\a}) \geq b_1^{(2)}(\tixc,\ti{\a}).\]
As we saw above,  $(\yc,\phi)$ fibers over $S^1$. It follows that
$\tibxc$ is homotopy equivalent to the fiber $\S$ of the fibration
and we see that $b_1^{(2)}(\tibxc,\ti{\a})=b_1^{(2)}(\S,\ti{\a})$.
Since $\S$ is a compact surface with boundary it follows immediately from Lemma \ref{lem:propb2} that
\[ b_1^{(2)}(\S,\ti{\a})=-\chi(\S)+b_0^{(2)}(\S,\ti{\a})\leq -\chi(\S)+1.\]
It therefore remains to compute $\chi(\S)$.

We denote the fibers of the fibrations $X(L_i)\to S^1$ by $\S_i$
and we denote the fibers of the fibrations $X(F_i)\to S^1$ by $F_i'$. Recall that $F_i'$ is diffeomorphic to $F_i$.
Note that $\S$ is the result of gluing the set of fibers $\{\S_i\}$
and the surfaces $\{F_i'\}$ along the longitudes of the links $L_i$.
Since the Euler
characteristic of the longitudes are zero  we get
\[ \chi(\S)=\sum_{i=1}^s \chi(\S_i)+\sum_{i=1}^r \chi(F_i).\]
By \cite[p.~78]{Di92} we have $\chi(\S_i)=1-\mu(\scrc,P_i)$ where $\mu(\scrc,P_i)$ denotes the Milnor number of the singularity $P_i$.

Now let $\scrd$ be the projective completion of $\scrc$.
Topologically $\scrd$ is given by adding disks to the boundary
components of $\scrc$  at ``infinity". Since $\scrc$ has degree $d$
and is in general position at infinity, there are exactly $d$ such
components. Since gluing in a disk increases the Euler
characteristic by $1$ we get that
\[
 \chi(\scrd)=\chi(\scrc)+ d.
 \]
Recall that the normalization of $\scrd$ is defined to be the curve $\hat{\scrd}$ without singularities obtained from
$\scrd$ by blow--ups.  Note that $\chi(\hat{\scrd})$ can be computed as follows: Let
$\scrd'$ be the result of first removing balls around the singularities, and let
$\scrd''$ be the result of gluing in disks to all the boundary components of
$\scrd'$. Then $\scrd''$ is topologically equivalent to $\scrd$ blown up at
the singularities, in particular
\[ \chi(\hat{\scrd})=\chi(\scrd'').\]
Since gluing in a disk increases the Euler characteristic by $1$ we
also get that
\[
\chi(\hat{\scrd})=\chi(\scrd')+ b_0(\partial \scrd').
\]
In our situation we therefore get
\[ \chi(\hat{\scrd})=\sum_{i=1}^r\chi(F_i)+\sum_{i=1}^sn_i+d.
\]
Summarizing we therefore see that
\[ \ba{rcl} b_1^{(2)}(\tixc)&\leq& b_1^{(2)}(\tibxc)\\
&\leq &-\chi(\S)+1\\
&=&-\sum_{i=1}^s \chi(\S_i)-\sum_{i=1}^r \chi(F_i)+1\\
&=&\sum_{i=1}^s \left(\mu(\scrc,P_i)-1\right)-\chi(\hat{\scrd})+ \sum_{i=1}^sn_i+d+1\\
&\leq&\sum_{i=1}^s
\left(\mu(\scrc,P_i)+n_i-1\right)+2g(\hat{\scrd})+d.\ea
\]
This completes  the proof of Theorem \ref{mainthm2}. We conclude with two remarks.

\begin{remark}
\bn
\item
In the case that $\G$ is a LITFA group we saw in Proposition
\ref{prop:b2kg} that the $L^2$--Betti numbers are determined by
ranks of homology modules over skew fields. In that case the
flatness of certain rings involved shows that statement of Theorem
\ref{mainthm} is an immediate consequence of Theorem \ref{mainthm2}
(we refer to  \cite{LM06} for details). This approach does not seem
to work if $\G$ is not a LITFA group.
\item Our methods carry over to prove generalizations of
Theorem 4.5, Theorem 4.7 and Corollary 4.8 in \cite{LM06}. We leave
the task of formulating and proving the precise statements to the
reader.
\item Given a knot $K$ we denote by $X(K)=S^3\sm \nu K$ its exterior, and by $\widetilde{X(K)}$ the infinite cyclic cover of $X(K)$. In the case that $K$ is a non--trivial fibered knot it follows from the above that
    $b_1^{(2)}\big(\widetilde{X(K)},\id\big)=2\mbox{genus}(K)-1$. Given any non--trivial knot $K$ we write $\ti{\pi}=\pi_1\big(\widetilde{X(K)}\big)$. By Proposition \ref{prop:b2kg} the sequence of $L^2$--Betti numbers $b_1^{(2)}\big(\widetilde{X(K)},\ti{\pi}\to \ti{\pi}/\ti{\pi}^{(n)}\big), n\geq 1$ equals the
    sequence of Cochran invariants $\delta_n(K)$, which was shown in \cite{Co04} to be a never--decreasing
    sequence of invariants which all give lower bounds on $2\,\mbox{genus}(K)-1$.
    Cochran's result can be interpreted as saying that the $L^2$--Betti number corresponding to `bigger' (PTFA--) quotients of $\ti{\pi}$ give better bounds on $2\,\mbox{genus}(K)-1$.
    It therefore seems natural to us to  conjecture that `in the limit' we get an equality, i.e. that $b_1^{(2)}\big(\widetilde{X(K)},\id\big)=2\,\mbox{genus}(K)-1$.
\en
\end{remark}


\begin{thebibliography}{10}
\bibitem[BH72]{BH72} R. G. Burns, V. W. Hale, {\em A note on group rings of certain torsion-free groups}, Canad. Math. Bull. 15, 441--445 (1972)
\bibitem[Co04]{Co04}
T. Cochran,  {\em Noncommutative knot theory}, Algebr. Geom. Topol. \textbf{4} (2004), 347--398.
\bibitem[COT03]{COT03} T.Cochran, K. Orr, P. Teichner, {\em Knot concordance, Whitney towers and $L\sp 2$-signatures}, Ann. of Math. (2)  157,  no. 2:
433--519 (2003)
\bibitem[CF07]{CF07}
J. I. Cogolludo, V. Florens, {\em Twisted Alexander polynomials of plane algebraic curves},
Journal of London Math. Society (2007) Vol. 76 Part 1: 105-121
\bibitem[DJL07]{DJL07}
M. W. Davis, T. Januszkiewicz, I. J. Leary, {\em The $l^2$--cohomology of hyperplane complements}, Groups, Geometry and Dynamics 1 (2007) 301–309.


\bibitem[Di92]{Di92} A. Dimca, {\em Singularities and topology of hypersurfaces}, Universitext,
Springer-Verlag, New York (1992)
\bibitem[DM07]{DM07} A. Dimca, L. Maxim, \emph{Multivariable Alexander invariants of hypersurface complements},
Trans. Amer. Math. Soc. 359 (2007), 3505--3528.


\bibitem[DLMSY03]{DLMSY03}
J. Dodziuk, P. Linnell, V. Mathai, T. Schick, S. Yates, {\em Approximating $L^2$-invariants, and
the Atiyah conjecture}, Preprint Series SFB 478 Muenster, Germany. Communications on Pure and
Applied Mathematics, vol. 56, no. 7:839-873 (2003)
\bibitem[Ha05]{Ha05}
S. Harvey, {\em Higher--order polynomial invariants of 3--manifolds giving lower bounds for the
Thurston norm}, Topology 44: 895--945 (2005)
\bibitem[Ha08]{Ha08}
S. Harvey, {\em
Homology Cobordism Invariants and the Cochran-Orr-Teichner Filtration of the Link Concordance Group},
Geom. Topol., Vol 12 (2008), 387--430.
%
\bibitem[LM06]{LM06}
C. Leidy, L. Maxim, {\em Higher-order Alexander invariants of plane algebraic curves},
 IMRN, Volume 2006 (2006), Article ID 12976, 23 pages.
 \bibitem[LM07]{LM07}
C. Leidy, L. Maxim, {\em
  Obstructions on fundamental groups of plane curve complements}, to be published
by the Proceedings of the IX Workshop on Real and Complex Singularities, Sao Carlos, 2006.
\bibitem[Lib82]{Lib82}
A. Libgober, {\em Alexander polynomial of plane algebraic curves and cyclic multiple planes}, Duke
Math. Journal 49 (1982), no. 4, 833--851.
\bibitem[Lib83]{Lib83} A. Libgober, \emph{Alexander invariants of plane algebraic
curves}, Singularities, Proc. Symp. Pure Math., Vol. 40(2), 1983,
135-143.
\bibitem[Lib86]{Lib86}
A. Libgober, {\em On the homotopy type of the complement to plane algebraic curves}, J. Reine
Angew. Math. 367 (1986), 103--114.
\bibitem[Lib92]{Lib92} A. Libgober, \emph{On the homology of finite abelian
covers}, Topology and its applications, \textbf{43} (1992) 157-166.
\bibitem[Lib94]{Lib94} A. Libgober, \emph{Homotopy groups of the complements to singular hypersurfaces, II},
Annals of Mathematics, 139 (1994), 117--144.
\bibitem[Lib01]{Lib01} A. Libgober, {\em Characteristic varieties of algebraic curves},
applications of algebraic geometry to coding theory, physics and
computation (Eilat 2001), Kluwer Acad. Publ. Dordrecht 2001,
215--254.
\bibitem[Lin92]{Lin92} P. Linnell, {\em Zero divisors and $L\sp 2(G)$}, C. R. Acad. Sci. Paris Ser. I Math. 315, no. 1, 49--53 (1992)
\bibitem[L\"u02]{Lu02} W. L\"uck, {\em $L\sp 2$-invariants: Theory and Applications to Geometry and
$K$-Theory}, Ergebnisse der Mathematik und ihrer Grenzgebiete. 3.
Folge. A Series of Modern Surveys in Mathematics, 44.
Springer-Verlag, Berlin, 2002.

\bibitem[Ma06]{Ma06} L. Maxim, \emph{Intersection homology and Alexander
modules of hypersurface complements}, Comm. Math. Helv. 81 (1),
123-155, 2006.
\bibitem[Mi63]{Mi63}
J. Milnor, {\em Morse theory}, Annals of Mathematics Studies, No. 51, Princeton University Press, Princeton, N.J. (1963)
\bibitem[Mi68]{Mi68} J. Milnor, {\em Singular points of complex hypersurfaces}, Annals of Mathematics Studies, No. 61,
Princeton University Press, Princeton, N.J. (1968)
\bibitem[Re98]{Re98}
H. Reich, {\em Group von Neumann algebras and related algebras}, Thesis, University of
G\"ottingen (1998)
\end{thebibliography}
\end{document}